\documentclass[10pt, reqno]{amsart}
\usepackage{graphicx, amssymb, amsmath, amsthm}
\numberwithin{equation}{section}

\usepackage{color}
\usepackage{epsfig}
\usepackage{subfigure}
\usepackage{tikz}
\usepackage[backref=page]{hyperref}

\hypersetup{urlcolor=blue, citecolor=red}

\usepackage{hyperref}

\let\Re=\undefined\DeclareMathOperator*{\Re}{Re}
\let\Im=\undefined\DeclareMathOperator*{\Im}{Im}

\newcommand{\R}{\mathbb{R}}
\newcommand{\C}{\mathbb{C}}

\newcommand{\Z}{\mathbb{Z}}

\newcommand{\HH}{\mathcal{H}}

\newcommand\A{\bf A}

\newcommand\D{\mathcal{D}}

\newtheorem{theorem}{Theorem}[section]

\newtheorem{lemma}[theorem]{Lemma}

\newtheorem{proposition}[theorem]{Proposition}

\theoremstyle{definition}
\newtheorem{definition}[theorem]{Definition}
\newtheorem{remark}[theorem]{Remark}

\makeatletter
\newcommand{\Extend}[5]{\ext@arrow0099{\arrowfill@#1#2#3}{#4}{#5}}
\makeatother

\begin{document}
\title[Dirac equation in magnetic fields]{Decay estimates for massive Dirac equation in a constant magnetic field}

\author{Zhiqing Yin}
\address{Department of Mathematics, Beijing Institute of Technology, Beijing 100081;}
\email{zhiqingyin@bit.edu.cn}

\begin{abstract}
We study the deacy and Strichartz estimates for the massive Dirac Hamiltonian in a constant magnetic fields in $\R_t\times\R^2_x$:
\begin{equation*}
\begin{cases}
 i\partial_tu(t,x)-\mathcal{D}_Au(t,x)=0,\\
 u(0,x)=f,
 \end{cases}
\end{equation*}
where $\mathcal{D}_A=-i{\bf \sigma}\cdot (\nabla-i{\bf A}(x))+\sigma_3m$ with $m\geq0$ being the mass and $\sigma_i$ being the Dirac matrices and the potential
${\bf A}(x)=\frac{B_0}{2}(-x_2,x_1),\,B_0>0$.
In particular, we show the $L^1(\R^2)\to L^\infty(\R^2)$ type micro-localized decay estimates, for any finite time $T>0$, there exists a constant $C_T$ such that
\begin{equation*}
\begin{split}
\Big\|e^{it\mathcal{D}_{A}}&\varphi(2^{-j}|\mathcal{D}_{A}|)f(x)\Big\|_{[L^\infty(\R^2)]^2}\\&
\leq C_T 2^{2j}(1+2^{j}|t|)^{-\frac12}
\|\varphi(2^{-j}|\mathcal{D}_{A}|)f\|_{[L^1{(\R^2)]^2}}, \quad |t|\leq T,
\end{split}
\end{equation*}
 and we further prove
the local-in-time Strichartz estimates for the Dirac equations with this unbounded potential.
\end{abstract}

\maketitle

\begin{center}
 \begin{minipage}{120mm}
 \end{minipage}
 \end{center}

\section{Introduction and main results}

\subsection{The Motivations}
We are interested in the model of dispersive equations in the electromagnetic fields. The electromagnetic phenomena play a fundamental role in quantum mechanics
and the electromagnetic Schr\"odinger Hamiltonian reads as
\begin{equation}\label{H-AV}
H_{A, V}=-(\nabla-i{\bf A}(x))^2+V(x)=\nabla_A^2+V(x),\quad \text{in}\quad L^2(\R^n; \C)
\end{equation}
where $\nabla_A:=i\nabla+{\bf A}(x)$ is the magnetic gradient and the electric scalar potential $V: \R^n\to \R$ and the magnetic vector potential
\begin{equation}
{\bf A}(x)=(A^1(x),\ldots, A^n(x)): \, \R^n\to \R^n
\end{equation}satisfies the Coulomb gauge condition
\begin{equation}\label{div0}
\mathrm{div}\, A=0.
\end{equation}
In  three dimensions, the magnetic vector potential $A$ produces the magnetic field $B$ which is given by
\begin{equation}\label{B-3}
B(x)=\mathrm{curl} (A)=\nabla\times A(x).
\end{equation}
In general dimension $n\geq2$, $B$ should be regards as a the matrix-valued field $B:\R^n\to \mathcal{M}_
{n\times n}(\R)$ given by
\begin{equation}\label{B-n}
B:=DA-DA^t,\quad B_{ij}=\frac{\partial A^i}{\partial x_j}-\frac{\partial A^j}{\partial x_i}.
\end{equation}
The Schr\"odinger operators with electromagnetic potentials have been extensively studied from the aspects of spectral and scattering theory. Avron-Herbest-Simon \cite{AHS1,AHS2,AHS3} and Reed-Simon \cite{RS} have discussed many important physical potentials (in particular the constant magnetic field and the Coulomb electric potential). However, the Hamiltonian \eqref{H-AV} can not explain finer electromagnetic effects very well due to the lack of an inner structure of electrons, for example, the \emph{spin}.\vspace{0.2cm}

For a clear understanding of finer electromagnetic effects, we give a more detailed analysis of the Dirac operator in $[L^2(\R^n;\C^N)]^2$ describing a \emph{spin-$\frac12$} charged particle under the influence of the constant magnetic fields, where $N=2^{\lceil(n+1)/2\rceil}$. It is a standard way to demonstrate its interaction with a particle by replacing the derivatives $\partial_k$ with their covariant counterpart $\partial_k-iA^k$ as follows:
\begin{equation}\label{D-AV}
\begin{split}
\mathcal{D}^m_{A, V}&=\begin{cases}-i\sum_{k=1}^3\alpha_k(\partial_k-iA^k(x))+\beta m+V(x)M_{4\times 4}, &n=3,\\
-i\sum_{k=1}^2\sigma_k(\partial_k-iA^k)+\sigma_3m+V(x)M_{2\times2},&n=2,
\end{cases}
\end{split}
\end{equation}
where ${\bf A}=(A^1(x),...,A^n(x)):\R^n\rightarrow\R^n$, $M_{N\times N}$ is a $N\times N$ complex matrix on $\R^n$, ${\bf\alpha}=(\alpha_1,\alpha_2, \alpha_3)$ and $\beta$ are $N\times N$ Dirac matrices
which satisfy the following \emph{canonical anticommutation relations}
\begin{equation}\label{alpha-r}
\begin{split}
\alpha_j\alpha_k+\alpha_k\alpha_j&=2\delta_{jk} I_{N},\quad 1\leq j,k\leq n,\\
\alpha_j\beta+\beta\alpha_j&=0,\quad j=1,2,\cdots,n\\
\beta^2&=I_N.
\end{split}
\end{equation}
Here $I_{N}$ is the $N\times N$ identity matrix and $\delta_{jk}$ denotes the Kronecker symbol ($\delta_{jk}=1$ if $j=k$; $\delta_{jk}=0$ if $j\neq k$).\vspace{0.2cm}

For $n=2$, there are at most three linearly independent anticommuting matrices: the Pauli matrices
\begin{equation}
\sigma_1=\left(\begin{array}{cc}0 & 1 \\1 & 0\end{array}\right),\quad
\sigma_2=\left(\begin{array}{cc}0 &-i \\i & 0\end{array}\right),\quad
\sigma_3=\left(\begin{array}{cc}1 & 0\\0 & -1\end{array}\right).
\end{equation}

Therefore, when $n=2$, the magnetic Dirac operator \eqref{D-AV} becomes
\begin{equation}\label{D-AV2}
\begin{split}
&\mathcal{D}^m_{A, V}=-i\sum_{k=1}^2\sigma_k(\partial_k-iA^k(x))+\sigma_3 m+V(x)M_{2\times 2}\\
  &=-\begin{pmatrix}
    -m & (i\partial_{1}+A^{1}-i(i \partial_{2}+A^{2})\\
    (i\partial_{1}+A^{1}+i(i \partial_{2}+A^{2}) &  m
  \end{pmatrix}+V(x)M_{2\times 2}.
\end{split}
\end{equation}
The system associated with \eqref{D-AV} is called massless if $m=0$, otherwise it is called massive. Moreover, the canonical anticommutation relations implies the free Dirac operator satisfies
 \begin{align*}
 \D^2=-\Delta I_N.
 \end{align*}
Then the Dirac equation can be listed within a diagonal system of wave equations
 \begin{align*}
 (i\partial_t+\D+m\sigma_3)(i\partial_t-\D-m\sigma_3) =(-\partial_{tt}+\Delta-m^2)I_{N}.
 \end{align*}
When a magnetic potential participates in, an electromagnetic wave equation will be a naturally result from the above reduction (discussed in \cite[Chapter 4]{LL}).
In the two dimension, there are two interesting magnetic fields in physics.

 $\bullet$ The first one is the magnetic potential ${\bf A}_B(x)=(A^1(x),A^2(x))$ which is given by
the {\em Aharonov-Bohm magnetic field} (see \cite{AB59}), that is, for magnetic flux $\alpha\in\R$

\begin{equation}\label{AB-potential}
{\bf A}_B(x)=\alpha\Big(-\frac{x_2}{|x|^2},\frac{x_1}{|x|^2}\Big),\quad x=(x_1,x_2)\in\mathbb{R}^2\setminus\{0\}.
\end{equation}
The Dirac equation in Aharonov-Bohm magnetic fields has been studied by F. Cacciafesta, P. D'Ancona, Zhang and the author \cite{CDYZ},
in which we prove the weighted dispersive and Strichartz estimates.

$\bullet$ The second magnetic potential ${\bf A}(x)=\vec{A}(x)=(A^1(x),A^2(x))$ is given by
the {\em constant magnetic field} (also called as uniform magnetic field), that is,
\begin{equation}\label{A-hmf}
{\bf A}(x)=\frac{B_0}{2}(-x_2,x_1),\quad B_0>0,
\end{equation}
 which is quite different from the Aharonov-Bohm magnetic fields studied in \cite{CDYZ}. We stress here the model is on the plane $\R^2$, where the magnetic field $B$ is given by
\begin{equation}\label{B-n}
B(x):=DA-DA^t,\quad B_{ij}=\frac{\partial A^i}{\partial x_j}-\frac{\partial A^j}{\partial x_i},\quad i,j=1,2.
\end{equation}
As in \cite{FV}, we define the vector field $B_\tau: \R^2\to \R^2$ as follows:
\begin{equation}
B_{\tau}=\frac x{|x|} B.
\end{equation}
Observe that $B_{\tau}\cdot x=0$, hence $B_{\tau}$ is a tangential vector field.
Obviously, the potential ${\bf A}(x)$ is unbounded at infinity and the magnetic filed $B_\tau(x)\neq 0$ leads to a trapped well.
So there is no hope to prove global-in-time dispersion estimates. \vspace{0.2cm}

In the present paper, we aim to study the decay estimates and Strichartz estimates of 2-D Dirac equation under the uniform magnetic fields, which give rise to very interesting phenomena and strange spectral properties of Dirac operators.
From the physical point of view the Dirac equation describes the spin of a particle and its magnetic moment in a completely natural way due to the influence of an external potential in a relativistically invariant manner.\vspace{0.2cm}

The decay estimates and Strichartz estimates of the dispersive equations has a long history due to their significance in analysis and PDEs fields.
Especially the dispersive equations with the Aharonov-Bohm potential, as a diffraction physic model and scaling critical purely magnetic potential, has attracted more and more people to study from mathematic viewpoint.
In particular, we refer to \cite{GV, KT} for a well established theory of the constant coefficient Schr\"odinger and the wave equations.
We refer to \cite{cacfan2, DF, DFVV, EGS1, EGS2, F, FFFP, FFFP1, GYZZ22, S, CYZ, CS, FZZ} for the Schr\"odinger, wave and Klein-Gordon equations with electromagnetic potentials in mathematic and physical fields. However,
for the Dirac equation,  the picture is far from complete. To the best of our knowledge, due to the rich algebraic structure of the Dirac equation, there only few available results in this direction are provided in \cite{cacser,cacserzha} proved the local smoothing and Strichartz estimates in the cases of the Coulomb potential perturbation and \cite{cacfan} for the Aharonov-Bohm magnetic field. Very recently, Danesi \cite{D} has established the Strichartz estimates
 for Dirac-Coulomb equation with a loss of angular regularity. We also mention that our work is very different from the works in \cite{EG, EGG, EGT},
for the perturbation approach there breaks down for our model since the magnetic potential is unbounded. \vspace{0.2cm}

\subsection{The main results}In this paper, we focus on the massive Dirac operator with the uniform (constant) magnetic field, that is,
\begin{equation}\label{D-A2}
\begin{split}
\mathcal{D}_A
  &=-\begin{pmatrix}
    -m & (i\partial_{1}+A^{1}-i(i \partial_{2}+A^{2})\\
    (i\partial_{1}+A^{1}+i(i \partial_{2}+A^{2}) &  m
  \end{pmatrix}\\
  &:=-\begin{pmatrix}
  -m&\D_+\\ \D_-&m\end{pmatrix},
\end{split}
\end{equation}
where ${\bf A}(x)=\vec{A}(x)=(A^1(x),A^2(x))$ is in \eqref{A-hmf}.
The purpose of this paper is to study the dispersion behavior of the solution of the Cauchy problem
\begin{equation}\label{eq:Dirac}
\begin{cases}
\displaystyle
 i\partial_tu=\mathcal{D}_Au,\quad u(t,x):\mathbb{R}_t\times\mathbb{R}_x^2\rightarrow\mathbb{C}^{2},\\
u(0,x)=f(x),
\end{cases}
\end{equation}
where $u(t,x)$ are vector-valued wavefunctions
\begin{equation}
u(t,x) =\begin{pmatrix}
    u_1(t,x) \\
      u_{2}(t,x)
    \end{pmatrix}\in \C^2.
\end{equation}


Before stating our main results, let us introduce some preliminary notations.
We define the magnetic Besov spaces as follows.  Let $\varphi\in C_c^\infty(\mathbb{R}\setminus\{0\})$, with $0\leq\varphi\leq 1$, $\text{supp}\,\varphi\subset[1/2,1]$, and
\begin{equation}\label{LP-dp}
\sum_{j\in\Z}\varphi(2^{-j}\lambda)=1,\quad \varphi_j(\lambda):=\varphi(2^{-j}\lambda), \, j\in\Z,\quad \phi_0(\lambda):=\sum_{j\leq0}\varphi(2^{-j}\lambda).
\end{equation}

\begin{definition}[ Magnetic Besov spaces]\label{def:besov} For $s\in\R$ and $1\leq p,r<\infty$, the homogeneous Besov norm of $\|\cdot\|_{\dot{\mathcal{B}}^s_{p,r}(\R^2)}$ is defined by
\begin{equation}\label{Besov}
\begin{split}
\|f\|_{{\mathcal{B}}^s_{p,r}(\R^2)}&=\Big(\|\phi_0(\sqrt{H_{B_0}})f\|_{L^p(\R^2)}^r+\sum_{j\geq 1}2^{jsr}\|\varphi_j(\sqrt{H_{B_0}})f\|_{L^p(\R^2)}^r\Big)^{1/r},
\end{split}
\end{equation}
where $H_{B_0}=(i\nabla+{\bf A}(x))^2$. In particular, $p=r=2$, we denote the Sobolev norm
\begin{equation}\label{Sobolev1}
\begin{split}
\|f\|_{ \dot{\HH}^s_{{\bf A}}(\R^2)}:=\|f\|_{\dot{\mathcal{B}}^s_{2,2}(\R^2)},\quad \|f\|_{ \HH^s_{\bf A}(\R^2)}:=\|(H_{B_0}+m^2+ B_0)^{\frac s2}f\|_{L^2(\R^2)}=\|f\|_{\mathcal{B}^s_{2,2}(\R^2)}.
\end{split}
\end{equation}

\end{definition}

\begin{remark}  Alternatively,  the Sobolev space is defined by
\begin{align}\label{def:sobolev}
 \dot \HH^{s}_{\bf A}(\R^2):= H_{B_0}^{-\frac s2}L^2(\R^2),
\end{align}
with norm
\begin{equation}\label{Sobolev2}
\begin{split}
\|f\|_{\dot \HH^s_{\bf A}(\R^2)}&:=\big\|H_{B_0}^{\frac s2}f\big\|_{L^2(\R^2)}.
\end{split}
\end{equation}
The two norms in \eqref{Sobolev1} and \eqref{Sobolev2} are equivalent, see Proposition \ref{prop:E-S-B} below.
\end{remark}

\begin{definition}\label{ad-pair}
The pair $(q,p)\in [2,\infty]\times [2,\infty)$ is said to be admissible,   if $(q,p)$ satisfies
\begin{equation}\label{adm}
\frac{2}q\leq\frac{1}2-\frac1p.
\end{equation}
For $s\in\R$, we denote $(q,p)\in \Lambda^W_{s}$ if $(q,p)$ is admissible and satisfies
\begin{equation}\label{scaling}
\frac1q+\frac {2}p=1-s.
\end{equation}

\end{definition}

Now we state our main result.

\begin{theorem}\label{thm:dispersive} Let ${\bf A}(x)$ be in \eqref{A-hmf} and define the Landau Hamiltonian $H_{B_0}$ by
\begin{equation}\label{op:HB}
H_{B_0}=(i\nabla+{\bf A}(x))^2.
\end{equation}
Let $u(t,x)$ be the solution of \eqref{eq:Dirac} and assume that the initial data $f$ satisfies $\varphi(2^{-j}\sqrt{H_{B_0}})f(x)\in  [L^1(\R^2)]^2$ where $\varphi$ is in \eqref{LP-dp}.  For any $t\in I:=[0,T]$ with any finite $T$,
then there exists a constant $C_T$ depending on $T$ such that
\begin{equation}\label{dis-w}
\begin{split}
\Big\|&\varphi(2^{-j}\sqrt{H_{B_0}}) u(t,x)\Big\|_{[L^\infty(\R^2)]^2}\\&
\leq C_T 2^{2j}(1+2^{j}|t|)^{-\frac12}
\|\varphi(2^{-j}\sqrt{H_{B_0}})f\|_{[L^1{(\R^2)]^2}}, \quad t\in I.
\end{split}
\end{equation}
Let $u(t,x)$ be the solution of \eqref{eq:Dirac} with $m>0$ and the initial data $f\in [\dot\HH_{\bf A}^s(\R^2)]^2$,  then the Strichartz estimates hold
\begin{equation}\label{stri:w}
\begin{split}
&\|u(t,x)\|_{[L^q(I;L^p(\R^2))]^2}\leq C_T \|f\|_{ [\HH_{\bf A}^s(\R^2)]^2},
\end{split}
\end{equation}
 where the pair $(q,p)\in \Lambda^W_{s}$ with $0\leq s\leq1$.
 
\end{theorem}

We sketch of the proof here. In contrast to \cite{CDYZ}, due to the fact that ${\bf A}(x)\in  L^2_{\text{loc}}(\R^2;\R^2)$, our idea here use the spinorial null structure to relate the Dirac equation to the Klein-Gordon models.
 By squaring the magnetic perturbed Dirac operator, we obtain
\begin{equation}\label{equ:WA}
(\mathcal{D}_A)^2=\begin{pmatrix}
    H_{B_0}+m^2-B_0 & 0\\
   0 & H_{B_0}+m^2+B_0
  \end{pmatrix}
\end{equation}
where
\begin{equation}
H_{B_0}=(i\nabla+{\bf A}(x))^2.
\end{equation}
If $u(t,x)$ solves \eqref{eq:Dirac},  since $(i\partial_t+\mathcal{D}_A)(i\partial_t-\mathcal{D}_A) =-\partial_t^2-(\mathcal{D}_A)^2$, then
$u(t,x)=\Big(\begin{smallmatrix}
    u^1(t,x) \\
      u^{2}(t,x)
    \end{smallmatrix}\Big)$ solves the equations
\begin{equation}\label{equ:W1}
\begin{cases}
\partial_{tt} u^1(t,x)+\big(H_{B_0}+m^2-B_0\big) u^1(t,x)=0,\\
u^1(0,x)=f^+(x),\\
\partial_t u^1(0, x)=i\mathcal{D}_+ f^-(x)-mf^+,
\end{cases}
\end{equation}
and
\begin{equation}\label{equ:W2}
\begin{cases}
\partial_{tt} u^2(t,x)+\big(H_{B_0}+m^2+B_0\big) u^2(t,x)=0,\\
u^2(0,x)=f^-(x),\\
\partial_t u^2(0, x)=i\mathcal{D}_- f^+(x)+mf^-.
\end{cases}
\end{equation}
Therefore we are reduce to study two wave equations, in which the component $H_{B_0}+m^2-B_0$ in \eqref{equ:WA} denotes the restriction of the Pauli operator on the spin-up subspace, while $H_{B_0}+m^2+B_0$ stands for the restriction on the
spin-down subspace.

We typically consider the above equations as an initial value problem in the Hilbert space
 \begin{align}
 \HH_{\bf A}^1(\R^2;\C)=\{f\in L^2{(\R^2;\C)}:(i\nabla+{\bf A})f\in L^2(\R^2;\C^2) \}
 \end{align}
and
 \begin{align}
 \HH_{\bf A}^2(\R^2;\C)=\{f\in H_A^1(\R^2;\C):(i\nabla+{\bf A})^2f\in L^2(\R^2;\C^2) \}
 \end{align}
If ${\bf A}\in L^2_{loc}(\R^2;\R^2)$ and div${\bf A}\in L^2(\R^2)$, then $(i\nabla+\bf A)^2$ is essentially self-adjoint on $C^\infty_0(\R^2)$ by the Leinfelder-Simader theorem \cite{LS}. The Pauli operator $P(A)$ is also essentially self-adjoint on $C^\infty_0(\R^2;\C^2)$ under mild regularity assumptions on $B$. In physics the Pauli and Dirac operators represent the non-relativistic and relativistic quantum mechanical Hamiltonians of a spin $1/2$ particle confined to a plane and interacting with the magnetic field $(0, 0, B)$.
\begin{remark}
The operator $H_{B_0}$ is self-adjoint and
the eigenvalues of $H_{B_0}$ are given by
\begin{equation*}
\lambda_{k}=(2k+1)B_0,\quad k\in\mathbb{N}.
\end{equation*}
The two operators $H_{B_0}+m^2\mp B_0$ are positive operators for any nonnegative mass $m\geq0$.
Therefore, the systems \eqref{equ:W1} and \eqref{equ:W2} can be regards as Klein-Gordon type equations respectively, hence we need to study
the propagator
$$e^{it\sqrt{H_{B_0}+m^2\mp B_0}},$$
which is much difficult than the Pauli equation considered in \cite{KL}. To this end, as done in our previous paper \cite{WZZ2}, we use the Bernstein inequality to deal with the low frequency.
For the high frequency, we use the classical subordination formula
\begin{equation*}
e^{-y\sqrt{H_{B_0}+m^2\mp B_0}}=\frac{y}{2\sqrt{\pi}}\int_0^\infty e^{-s(H_{B_0}+m^2\mp B_0)}e^{-\frac{y^2}{4s}}s^{-\frac{3}{2}}ds,\quad y>0,
\end{equation*}
which provides  a connecting bridge between the Schr\"odinger propagator and the half-wave propagator.
\end{remark}
\vspace{0.1cm}

The paper is organized as follows. In Section \ref{sec:pre}, as a preliminary step, we recall the self-adjoint extension of $\D_A$ and the spectrum of the operator $H_{B_0}$, and prove the equivalence between Sobolev norm and a special Besov norm. In Section \ref{sec:pre}, we construct the heat kernel and prove the Gaussian upper bounds and we modify the subordination formula about $e^{it\sqrt{H_{B_0}+m^2\mp B_0}}$. In Section \ref{sec:tool}, we prove the Bernstein inequalities and the square function inequality
by using the heat kernel estimates.  Finally, in Section \ref{sec:decay} and Section \ref{sec:str}, we prove the dispersive estimate \eqref{dis-w} and the Strichartz estimate \eqref{stri:w} in Theorem \ref{thm:dispersive} respectively.
\vspace{0.2cm}

\section{the spectral theory of $\D_A$ and $H_{B_0}$} \label{sec:pre}

In this section, we first briefly prove that $\D_A$ is essentially self-adjoint and collect some harmonic analysis tools including Sobolev norm for the operator $H_{B_0}$
from \cite{WZZ1, WZZ2} by setting $\alpha=0$ and Dirac operator $\D_A$.

\subsection{Self-adjoint extension} We can use von Neumann theory to study the self-adjoint extension of $\D_A$. In fact, the operator $\D_A$ is essential self-adjoint due to
that the dimension of deficiency space $N_{\pm}=\ker(d_{k}^{*}\mp i)$ vanishes, where $d_{k}^*$ is the adjoint operator of $d_{k}$ and
\begin{equation*}
  d_{k}:=
  \begin{pmatrix}
    m & i(\partial_r+\frac{k+1}r-\frac {B_0}2 r) \\
    i(\partial_r-\frac{k}r+\frac {B_0}2 r) & -m
  \end{pmatrix}.
\end{equation*}
We refer to \cite[Section 2]{CDYZ}. For self-contained, we provide the proof of  the conclusion $N_{\pm}=\ker(d_{k}^{*}\mp i)=\emptyset$.
First, we notice the orthogonal decomposition
$\big[L^{2}(\mathbb{R}^{2})\big]^{2}$
\begin{equation}\label{eq:spDir}
  (L^{2}(\mathbb{R}^{2}))^{2}=
  \bigoplus_{k\in \mathbb{Z}}
  L^{2}(rdr)^{2}\otimes h_{k}(\mathbb{S}^{1}),
\end{equation}
where $h_{k}(\mathbb{S}^{1})$ is the one dimensional space
\begin{equation*}
  h_{k}=
  h_{k}(\mathbb{S}^{1})=
  \left[
    \begin{pmatrix}
      e^{ik \theta} \\
      e^{i(k+1)\theta}
    \end{pmatrix}
  \right]=
  \left\{
    c\begin{pmatrix}
      e^{ik \theta} \\
      e^{i(k+1)\theta}
    \end{pmatrix}:
    c\in \mathbb{C}
  \right\}.
\end{equation*}
Then, for any $f=\Big(\begin{smallmatrix} \phi \\ \psi \end{smallmatrix}\Big)
  \in \big[L^{2}(\mathbb{R}^{2})\big]^{2}$, we have
\begin{equation*}
  \phi=\sum_{k\in \mathbb{Z}}\phi _{k}(r)e^{ik \theta},
  \qquad
  \psi=\sum_{k\in \mathbb{Z}}\psi _{k}(r)e^{i(k+1)\theta}.
\end{equation*}
Hence, we obtain
\begin{equation*}
\begin{split}
  \mathcal{D}_{A}
  \sum_{k}
  \begin{pmatrix}
    e^{ik \theta}\ \phi_{k}(r) \\
    e^{i(k+1)\theta}\ \psi_{k}(r)
  \end{pmatrix}
 & = \sum_{k}d_k
  \begin{pmatrix}
    e^{ik \theta}\ \phi_{k}(r) \\
    e^{i(k+1)\theta}\ \psi_{k}(r)
  \end{pmatrix}\\
  &=
  \sum_{k}
  \begin{pmatrix}
    e^{ik \theta}[m\phi_k(r)+i\ (\partial_r+\frac{k+1}r-\frac {B_0}2 r)\psi_{k}(r)] \\
    e^{i(k+1)\theta}[-m\psi_k(r)+i\ (\partial_r-\frac{k}r+\frac {B_0}2 r)\phi_{k}(r)]
  \end{pmatrix}.
  \end{split}
\end{equation*}

To ensure the radial Dirac operator $d_k$ is symmetric and well defined, the domain is restricted on $C^\infty_0((0,\infty),rdr)\subset L^2((0,\infty),rdr)$.
It is easy to verify that $d_k$ is symmetric. In \cite{FP} (with $\kappa=0$), it implies that $d_k$ is essentially self-adjoint. In the following we give a simple proof that $d_k$ admits one-parameter self-adjoint extensions $\bar{d_k}$.
To see that the operator $\D_A$ is essentially self-adjoint, from von Neumann theory, we need to conclude the deficiency subspaces $N_{\pm}$ is empty.

\textsc{The closure $\overline{d_{k}}$}. The operator $d_k$ defined on $D(d_k)=[C^\infty_0(\R_+)]^2$, a dense subspace of $[L^2(\R_+,rdr)]^2$. For $f=\big(\begin{smallmatrix}\phi \\ \psi\end{smallmatrix}\big)\in \big[C_{0}^{\infty}((0,+\infty))\big]^{2}$, it is easy to seen that
\begin{equation}
 \begin{aligned}
 \|d_k f\|^2_{(L^2(rdr))^2}=\|f'\|_{L^2(rdr)}&+\Big\|\Big( \frac{k}r+\frac {B_0}2 r \Big)\phi\Big\|_{L^2(rdr)}+\|m\phi\|_{L^2(rdr)}\\&+
 \Big\|\Big(\frac{k+1}r+\frac{B_0}2 r\Big)\psi\Big\|_{L^2(rdr)}+\|m\psi\|_{L^2(rdr)}.
 \end{aligned}
 \end{equation}
From \cite[Appendix B]{FP}, we can easily get the functions in $D(\overline{d_{k}})$ are continuous near the origin and vanishing for $r\rightarrow 0^+$, that is
 \begin{align}
 D(\overline{d_k})=\{\phi\in L^2(rdr)| \phi', \phi/r \in L^2(rdr), \phi\in C([0,\infty)), \lim_{r\rightarrow0^+}\phi(0)=0 \}^2.
 \end{align}

\textsc{The adjoint $d_{k}^*$}. Let us recall the definition of domain of $d_{k}^*$,
 for $f=(\begin{smallmatrix}\phi \\ \psi\end{smallmatrix})
  \in D(d^{*}_{k})$
iff $\exists g\in L^{2}(rdr)^{2}$ such that
\begin{equation}\label{eq:adj}
  \langle f,d_{k}u\rangle=
  \langle g,u\rangle
  \qquad
  \forall u\in D(d_k).
\end{equation}
The adjoint $d_{k}^*$ is defined by $g=d_{k}^*f$. The weak derivative of $f$ is locally in $[L^{2}(rdr)]^{2}$ away from the origin. Making use of the Sobolev's lemma \cite{RS} and integrations by parts in \eqref{eq:adj}, one conclude $d_k^*$ acts as same expression as $d_k$.

\textsc{Deficiency indices of $\overline{d}_{k}$}. Since $\overline{d}_{k}$ is closed, densely defined and symmetric, von Neumann's theory applies. Now we compute the dimension of deficiency space $N_{+}=\ker(d_{k}^{*}- i)$ first. For $f=\big(\begin{smallmatrix} \phi \\ \psi \end{smallmatrix}\big)\in D(\ker(d_{k}^{*}))$,
 \begin{align}
 d_{k}^{*}f= i f,
 \end{align}
the coupled differential equations
 \begin{align}
 \begin{cases}
 i(\partial_r+\frac{k+1}r-\frac {B_0}2 r)\psi= (i-m)\phi,\\
 i(\partial_r-\frac{k}r+\frac {B_0}2 r)\phi= (i+m)\psi,
 \end{cases}
 \end{align}
this implies
 \begin{align}
 \begin{cases}
 [\partial_r^2+\frac1r\partial_r-(\frac{k^2}{r^2}-B_0(k+1)+\frac{B_0^2}{4}r^2)+m^2+1]\phi=0,\\
 [\partial_r^2+\frac1r\partial_r-(\frac{(k+1)^2}{r^2}-B_0k+\frac{B_0^2}{4}r^2)+m^2+1]\psi=0.\\
 \end{cases}
 \end{align}
One can solve this equation by
 \begin{align}
 \phi=e^{-\frac{B_0r^2}4}r^{|k|}F_1(\frac {B_0}2 r^2),\quad \psi=e^{-\frac{B_0r^2}4}r^{|k+1|}F_2(\frac {B_0}2r^2),
 \end{align}
since Kummer's equations (see \cite{Abra})
 \begin{align*}
 \begin{cases}
 \frac {B_0}2r^2\frac{d^2F_1(\frac {B_0}2r^2)}{d(r^2)^2}+[(|k|+1)-\frac {B_0}2r^2)]\frac{dF_1(\frac {B_0}2r^2)}{d(r^2)}+(\frac{k-|k|}2-\frac{m^2+1}{2B_0}) F_1(\frac {B_0}2r^2)=0,\\
 \frac {B_0}2r^2\frac{d^2F_2(\frac {B_0}2r^2)}{d(r^2)^2}+[(|k+1|+1)-\frac {B_0}2r^2)]\frac{dF_2(\frac {B_0}2r^2)}{d(r^2)}+(\frac{k+1-|k+1|-2}2-\frac{m^2+1}{2B_0}) F_2(\frac {B_0}2r^2)=0.
 \end{cases}
 \end{align*}
The first equation has two linearly independent solutions: $M(\frac{k-|k|}2-\frac{m^2+1}{2B_0}, |k|+1,  \frac{B_0}2r^2)$ and $U(\frac{k+1-|k+1|-2}2-\frac{m^2+1}{2B_0}, |k+1|+1,  \frac{B_0}2r^2)$, where
 \begin{align}
 M(a,b,s)=\sum^\infty_{n=0}\frac{(a)_n}{(b)_n}\frac{s^n}{n!}, \quad b\neq0,-1,-2,\cdot\cdot\cdot
 \end{align}
with $(a)_n (a\in\R)$ denoting the Pochhammer's symbol
 \begin{align}
 (a)_n=\begin{cases}1,& n=0;\\
 a(a+1)(a+2)\cdot\cdot\cdot(a+n-1),&n=1,2,\cdot\cdot\cdot\end{cases}
 \end{align}
and
 \begin{align*}
 U(a,b,s)=\frac{\Gamma(1-b)}{\Gamma(a+b-1)}M(a,b,s) +\frac{\Gamma(b-1)}{\Gamma(a)}s^{1-b}M(a-b+1,2-b,s).
 \end{align*}
From \cite{Abra}, to ensure that $\phi\in L^2((0,\infty), rdr)$, one requires $0<|k|+1<2$ (see \cite[P508]{Abra}).
Similarly for the second equation, to ensure that $\psi\in L^2((0,\infty), rdr)$, one requires $0<|k+1|+1<2$.
So there is no $k\in\Z$ satisfying the two requirements, therefore we conclude that $N_+=N_-=\emptyset$, that is to say the operator $d_k$ is essentially self-adjoint and its unique self-adjoint extension is given by the closure of its graph.

\subsection{The spectrum of the operator $H_{B_0}$}
From \eqref{A-hmf}, we observe the fact that ${\bf A}(x)\in  L^2_{\text{loc}}(\R^2;\R^2)$, hence
 $H_{B_0}$ is an essentially
self-adjoint operator on $L^2(\R^2; \C)$ with domain $\HH_A^2(\R^2; \C)$. The operator $H_{B_0}$  is named after Landau operator.
Our argument is based on the known results in \cite{WZZ1, WZZ2} which we restated here.
\begin{proposition}[The spectrum for $H_{B_0}$]\label{prop:spect}
The eigenvalues of $H_{B_0}$ are given by
\begin{equation}\label{eigen-v}
\lambda_{k,\ell}=(2\ell+1+|k|)B_0+kB_0,\quad \ell,\,k\in\mathbb{Z},\, \ell\geq0,
\end{equation}
and the (finite) multiplicity of $\lambda_{k,\ell}$ is
\begin{equation*}
\#\Bigg\{j\in\mathbb{Z}:\frac{\lambda_{k,\ell}-jB_0}{2B_0}-\frac{|j|+1}{2}\in\mathbb{N}\Bigg\}.
\end{equation*}
Furthermore, let $\theta=\frac{x}{|x|}$, the corresponding eigenfunction is given by
\begin{equation}\label{eigen-f}
V_{k,\ell}(x)=|x|^{|k|}e^{-\frac{B_0 |x|^2}{4}}\, P_{k,\ell}\Bigg(\frac{B_0|x|^2}{2}\Bigg)e^{ik\theta},
\end{equation}
where  $P_{k,\ell}$ is the polynomial of degree $m$ given by
\begin{equation*}
P_{k,\ell}(r)=\sum_{n=0}^\ell\frac{(-\ell)_n}{(1+|k|)_n}\frac{r^n}{n!}
\end{equation*}
with $(a)_n$ ($a\in\R$) denoting the Pochhammer's symbol
\begin{align*}
(a)_n=
\begin{cases}
1,&n=0;\\
a(a+1)\cdots(a+n-1),&n=1,2,\cdots
\end{cases}
\end{align*}\end{proposition}

\begin{proof} This is exact same to \cite[Proposition 2.1]{WZZ1} with $\alpha=0$.
\end{proof}

\begin{remark}
The spectrum of the operator $H_{B_0}+m^2\mp B_0$ is $\lambda_{k, \ell}+m^2\mp B_0\geq m^2\geq 0$.
This fact is important in the following argument in Section 3.
\end{remark}

\subsection{The Sobolev spaces}
From the above spectral property, we obtain the following proposition as a corollary of \cite[Proposition 2.5]{WZZ1}. The following lemma of the norm induced by $\D_A$ is one of key ingredients for proving our main result.
\begin{proposition}[Equivalent norms]\label{prop:E-S-B}
Let the Sobolev norm and Besov norm be defined in \eqref{Sobolev2} and \eqref{Besov} respectively. For $s\in\R$, then
there exist positive constants $c, C$ such that
\begin{equation}\label{S-B-H}
c\|\phi\|_{\dot \HH^s_{B_0}(\R^2)}\leq \|\phi\|_{\dot{\mathcal{B}}^s_{2,2}(\R^2)}\leq C \|\phi\|_{\dot \HH^s_{B_0}(\R^2)},
\end{equation}
and
\begin{equation}\label{S-B-H'}
c\|\phi\|_{ \HH^s_{B_0}(\R^2)}\leq \|\phi\|_{{\mathcal{B}}^s_{2,2}(\R^2)}\leq C \|\phi\|_{ \HH^s_{B_0}(\R^2)}.
\end{equation}
\end{proposition}

\begin{lemma}\label{lem:D-norm}
For any $s\in[0,1]$, we have
 \begin{align}
 \|\D_A^sf\|_{(L^2)^2}\leq  \|f\|_{(\HH^s_{\bf A})^2}
 \end{align}
where the fractional powers of $\D_A$ commute with the flow of \eqref{eq:Dirac}.
\end{lemma}
\begin{proof}It is sufficient to prove the case $s=1$, as the full range of exponents can be obtained by interpolation because the case $s=0$ is obvious.

Denote
 $$\nabla_A=(\partial^A_1,\partial^A_2)=(i\partial_1+iA^1(x),i\partial_2+A^2(x)),$$
hence
 \begin{align*}
 \|\nabla_Af\|_{L^2}&=\int_{\R^2}|(i\partial_1+A^1)f|^2+|(i\partial_2+A^2)f|^2dx\\
 &=\int_{\R^2}[|\partial_1f|^2+|A^1f|^2-2\Im(\partial_1fA^1f) +|\partial_2f|^2+|A^2f|^2-2\Im(\partial_2fA^2f)]dx.
 \end{align*}
 Without loss of generality, we consider the case $m=0$,
 then for $f=(\begin{smallmatrix}f^+ \\ f^-\end{smallmatrix})\in (C^\infty_0(\R^2))^2$ there exist possible constant $C$ such that
 \begin{align*}
 \|\D_Af\|^2_{[L^2]^2}&=\int_{\R^2}|(-i\partial^A_1-\partial^A_2)f^-|^2+ |(-i\partial^A_1+\partial^A_2)f^-|^2dx\\
 &=\int_{\R^2}[|\nabla_A f^-|^2-2\Re(\partial_1f^-A^2f^-)+2\Re(A^1f^-\partial_2f^-)\\&\qquad+|\nabla_Af^+|^2 +2\Re(\partial_1f^+A^2f^+)-2\Re(A^1f^+\partial_2f^+)d]x\\
 &=\int_{\R^2}[|\nabla_A f^-|^2-B_0|f^-|^2+|\nabla_Af^+|^2+B_0|f^+|^2]dx\\
 &\leq\|f^-\|_{\HH^1_{\bf A}}+\|f^+\|_{\HH^1_{\bf A}}=\|f\|_{[\HH^1_{\bf A}]^2}.
 \end{align*}

\end{proof}

\subsection{Littlewood-Paley theory}\label{sec:tool}
For this part, we modify the argument of \cite{MS} and \cite{DPR} to provide a the key formula to construct a representation of propagator $e^{it\sqrt{H_{B_0}+m^2\mp B_0}}$.
Comparing with \cite{WZZ1, WZZ2}, since $\alpha=0$,  we are easy to use the Mehler heat kernel of $e^{-tH_{0,B_0}}$ (e.g. \cite[P168]{Simon2}) to obtain the heat kernel of  $H_{B_0}$
\begin{equation}\label{equ:Meh}
e^{-tH_{B_0}}(x,y)=\frac{B_0}{4\pi\sinh(B_0t)}e^{-\frac{B_0|x-y|^2}{4\tanh(B_0t)}-\frac{iB_0}2(x_1y_2-x_2y_1)}.
\end{equation}
Therefore, we use the standard argument, as did in \cite{WZZ1, WZZ2}, to show the Bernstein inequalities and the Littlewood-Paley theory associated with the operator $H_{B_0}$.

\begin{proposition}[Bernstein inequalities]
Let $\varphi(\lambda)$ be a $C^\infty_c$ bump function on $\R$  with support in $[\frac{1}{2},2]$, then it holds for any $f\in L^q(\mathbb{R}^2)$ and $j\in\mathbb{Z}$
\begin{equation}\label{est:Bern}
\|\varphi(2^{-j}\sqrt{H_{B_0}})f\|_{L^p(\mathbb{R}^2)}\lesssim2^{2j\big(\frac{1}{q}-\frac{1}{p}\big)}\|f\|_{L^q(\mathbb{R}^2)},\quad 1\leq q\leq p\leq\infty.
\end{equation}
\end{proposition}

\begin{proposition}[The square function inequality]\label{prop:squarefun} Let $\{\varphi_j\}_{j\in\mathbb Z}$ be a Littlewood-Paley sequence given by\eqref{LP-dp}.
Then for $1<p<\infty$,
there exist constants $c_p$ and $C_p$ depending on $p$ such that
\begin{equation}\label{square}
c_p\|f\|_{L^p(\R^2)}\leq
\Big\|\Big(\sum_{j\in\Z}|\varphi_j(\sqrt{H_{B_0}})f|^2\Big)^{\frac12}\Big\|_{L^p(\R^2)}\leq
C_p\|f\|_{L^p(\R^2)}.
\end{equation}
\end{proposition}
\begin{proof}
By using the heat kernel \eqref{equ:Meh}, Proposition \ref{prop:squarefun} follows from the Rademacher functions argument in \cite{Stein}. We also refer the reader for result that the square function inequality \ref{square} can be derived from the heat kernel with the Gaussian upper bounds.
\end{proof}
\begin{remark}If we replace $H_{B_0}$ by $H_{B_0}+m^2\mp B_0$, we will see that
\begin{equation}\label{equ:Meh}
e^{-t(H_{B_0}+m^2\mp B_0)}(x,y)=\frac{B_0e^{-t(m^2\mp B_0)}}{4\pi\sinh(B_0t)}e^{-\frac{B_0|x-y|^2}{4\tanh(B_0t)}-\frac{iB_0}2(x_1y_2-x_2y_1)}.
\end{equation}

\end{remark}

\subsection{A key lemma}

To prove our results, we give the following proposition about $e^{it\sqrt{H_{B_0}+m^2\mp B_0}}$ which modified the subordination formula  (see \cite{DPR,MS,WZZ2}).
 \begin{align}
 e^{-y\sqrt{x}}=\frac{y}{2\sqrt{\pi}}\int^\infty_0 e^{-sx-\frac{y^2}{4s}}s^{-\frac32}d s,\quad x,y>0.
 \end{align}
\begin{proposition}
Given a fixed $T>0$ and if $\varphi(\lambda)\in C_c^\infty(\mathbb{R})$ is supported in $[\frac{1}{2},2]$, then, for all $j\in\Z, x>B_0,\tilde{x}:=x+m^2\mp B_0>0$ and $0<t\leq T$ with $2^jt\geq 1$,  we can write
\begin{equation}\label{key}
\begin{split}
&\varphi(2^{-j}\sqrt{x})e^{it\sqrt{\tilde{x}}}\\&=\varphi(2^{-j}\sqrt{x})\tilde{\rho}\big(\frac{t\tilde{x}}{2^j}, 2^jt\big)
+\varphi(2^{-j}\sqrt{x})\big(2^jt\big)^{\frac12}\int_0^\infty \chi(s,2^jt)e^{\frac{i2^jt}{4s}}e^{i2^{-j}ts\tilde{x}}\,ds,
\end{split}
\end{equation}
where $\tilde{\rho}(a, t)$ satisfies
\begin{equation}\label{est:rho}
\big|\partial_a^\alpha\partial_t^\beta\tilde{\rho}(a, t)\big|\leq C_{N,\alpha,\beta} (a+t)^{-N}, \quad \frac18\leq \frac at\leq 8, \,  t\geq 1, \forall N\geq 1,
\end{equation}
and  $\chi\in C^\infty(\mathbb{R}\times\mathbb{R})$ with $\text{supp}\,\chi(\cdot,\tau)\subseteq[\frac{1}{16},8]$ such that
\begin{equation}\label{est:chi}
\sup_{\tau\in\R}\big|\partial_s^\alpha\partial_\tau^\beta \chi(s,\tau)\big|\lesssim_{\alpha,\beta}(1+|s|)^{-\alpha},\quad \forall \alpha,\beta\geq0.
\end{equation}
\end{proposition}

\begin{remark}
If this has been done, then by the spectral theory for the non-negative self-adjoint operator $H_{B_0}$, we can have the representation of the micolocalized
half-wave propagator
\begin{equation}\label{key-operator}
\begin{split}
&\varphi(2^{-j}\sqrt{H_{B_0}})e^{it\sqrt{H_{B_0}+m^2\mp B_0}}\\&=\varphi(2^{-j}\sqrt{H_{B_0}})\tilde{\rho}\big(\frac{t(H_{B_0}+m^2\mp B_0)}{2^j}, 2^jt\big)\\
&+\varphi(2^{-j}\sqrt{H_{B_0}})\big(2^jt\big)^{\frac12}\int_0^\infty \chi(s,2^jt)e^{\frac{i2^jt}{4s}}e^{i2^{-j}ts(H_{B_0}+m^2\mp B_0)}\,ds.
\end{split}
\end{equation}
\end{remark}

\begin{proof}
Our starting point of the proof is the subordination formula
\begin{equation}\label{change:subord}
e^{-y\sqrt{\tilde{x}}}=\frac{y}{2\sqrt{\pi}}\int_0^\infty e^{-s\tilde{x}-\frac{y^2}{4s}}s^{-\frac{3}{2}}ds,\quad \tilde{x}=x+m^2\mp B_0,y>0.
\end{equation}
To obtain $e^{it\sqrt{\tilde{x}}}$, we extend \eqref{change:subord} by setting $y=\epsilon-it$ with $\epsilon>0$
\begin{equation}\label{identity}
\begin{split}
e^{it\sqrt{\tilde{x}}}&=\lim_{\epsilon\to 0^+}e^{-(\epsilon-it)\sqrt{\tilde{x}}}\\
&=\lim_{\epsilon\to 0}\frac{\sqrt{\epsilon-it}}{2\sqrt{\pi}}I_{\epsilon, \epsilon \tilde{x}}(t\tilde{x},t),
\end{split}
\end{equation}
where
\begin{equation*}
I_{\epsilon, \delta}(a,t):=\int_0^\infty e^{ira}e^{-\delta r}e^{\frac{it}{4r}}e^{-\frac{\epsilon}{4r}} r^{-\frac{3}{2}}dr.
\end{equation*}
By the dominate convergence theorem, we have that
\begin{equation*}
\begin{split}
e^{it\sqrt{\tilde{x}}}
&=\lim_{\epsilon\to 0}\frac{\sqrt{\epsilon-it}}{2\sqrt{\pi}}I_{\epsilon, \epsilon \tilde{x}}(t\tilde{x},t)=\sqrt{\frac{t}{4\pi}}e^{-\frac\pi 4i}\lim_{\epsilon\to0}I_{\epsilon, \epsilon \tilde{x}}(t
\tilde{x},t).
\end{split}
\end{equation*}
Thus it suffices to consider the oscillation integral
\begin{equation}\label{limit}
\begin{split}
\lim_{\epsilon\to 0}I_{\epsilon, \epsilon \tilde{x}}(a,t)=I_{0, 0}(a,t)=\int_0^\infty e^{ira} e^{\frac{it}{4r}} r^{-\frac{3}{2}}dr.
\end{split}
\end{equation}

\begin{lemma}\label{lem: stationary} Let
\begin{equation*}
\begin{split}
I(a,t)=\int_0^\infty e^{ira} e^{\frac{it}{4r}} r^{-\frac{3}{2}}dr.
\end{split}
\end{equation*}
Then we can write
\begin{equation}\label{est: stat}
I(a,t)=\tilde{\rho}(a, t)+\int_0^\infty e^{ira} e^{\frac{it}{4r}} \tilde{\chi} (r)\, dr,
\end{equation}
where $\tilde{\chi}(r)\in C_0^\infty(r)$ and $\mathrm{supp}\,\tilde{\chi} \subset [\frac1{16},8]$ and $\tilde{\rho}(a, t)$ satisfies
\begin{equation}\label{est:rho}
\big|\partial_a^\alpha\partial_t^\beta\tilde{\rho}(a, t)\big|\leq C_{N,\alpha,\beta} (a+t)^{-N}, \quad \frac18\leq \frac at\leq 8, \,  t\geq 1, \forall N\geq 1.
\end{equation}
\end{lemma}
\begin{proof}
This is \cite[Lemma]{WZZ2} which is the consequence of the stationary phase argument.
\end{proof}
By \eqref{identity} and \eqref{limit} and noticing
$$I(a,t)=2^{\frac j2}I(2^{-j}a, 2^jt),$$ we have
that \begin{equation*}
\varphi(2^{-j}\sqrt{x})e^{it\sqrt{\tilde{x}}}=\sqrt{\frac{t}{4\pi}}e^{-\frac\pi 4i}\varphi(2^{-j}\sqrt{x}) 2^{\frac j2} I\big(\frac{t\tilde{x}}{2^j}, 2^jt\big).
\end{equation*}
Therefore, by using this lemma, we prove \eqref{key}
\begin{equation*}
\begin{split}
&\varphi(2^{-j}\sqrt{x})e^{it\sqrt{\tilde{x}}}\\=&\frac1{\sqrt{4\pi}}e^{-\frac\pi 4i}\big(2^jt\big)^{\frac12}\varphi(2^{-j}\sqrt{x}) \Big(\tilde{\rho}\big(\frac{t\tilde{x}}{2^j}, 2^jt\big)
+\int_0^\infty \tilde{\chi}(s)e^{\frac{i2^jt}{4s}}e^{i2^{-j}t\tilde{x}s}\,ds\Big).
\end{split}
\end{equation*}

We need consider this expression when $2^jt\geq 1$. To this end, let $\phi\in C^\infty([0,+\infty)$ satisfies $\phi(t)=1$ if $t\ge1$ and $\phi(t)=0$ if $0\leq t\leq \frac12$,
then set $$\tilde{\rho}\big(\frac{t\tilde{x}}{2^j}, 2^jt\big)=e^{-\frac\pi 4i}\big(2^jt\big)^{\frac12}\varphi(2^{-j}\sqrt{x})\tilde{\rho}\big(\frac{t(x+m^2\pm B_0)}{2^j}, 2^jt\big)\phi(2^jt). $$
This together with \eqref{est:rho} shows
\begin{equation*}
\big|\partial_a^\alpha\partial_t^\beta{\rho}(a, t)\big|\leq C_{N,\alpha,\beta}(1+ (a+t))^{-N}, \quad  \forall N\geq 0.
\end{equation*}
which implies
$\rho(a,t)\in \mathcal{S}(\R_+\times\R_+)$.
Set $$\chi\big(s, 2^jt\big)=e^{-\frac\pi 4i}\tilde{\chi}\big(s\big)\phi(2^jt),$$
 then $\chi$ satisfies \eqref{est:chi}.
 then we finally write
\begin{equation*}
\begin{split}
&\varphi(2^{-j}\sqrt{x})e^{it\sqrt{\tilde{x}}}\\=&\tilde{\rho}\big(\frac{t\tilde{x}}{2^j}, 2^jt\big)
+\big(2^jt\big)^{\frac12}\varphi(2^{-j}\sqrt{x}) \int_0^\infty \chi(s,2^jt)e^{\frac{i2^jt}{4s}}e^{i2^{-j}t\tilde{x}s}\,ds,
\end{split}
\end{equation*}
which proves \eqref{key} as desired.
\end{proof}

\section{The proof of  Theorem \ref{thm:dispersive}}
In this section, we prove Theorem \ref{thm:dispersive}.

\subsection{The decay estimates}\label{sec:decay}
 We first prove the following results
\begin{proposition}\label{prop: decay} Let $2^{-j}|t|\leq \frac{\pi}{8B_0}$ and $\varphi$ be in \eqref{LP-dp}, then
\begin{equation}\label{est: mic-decay1}
\begin{split}
\big\|\varphi(2^{-j}\sqrt{H_{B_0}})&e^{it\sqrt{H_{B_0}+m^2\mp B_0}}f\big\|_{L^\infty(\R^2)}\\&\lesssim 2^{2j}\big(1+2^j|t|\big)^{-\frac12}\|\varphi(2^{-j}\sqrt{H_{B_0}}) f\|_{L^1(\R^2)}.
\end{split}
\end{equation}
In particular, for $0<t<T$ with any finite $T$, there exists a constant $C_T$ depending on $T$ such that
\begin{equation}\label{est: mic-decay2}
\begin{split}
\big\|\varphi(2^{-j}\sqrt{H_{B_0}})&e^{it\sqrt{H_{B_0}+m^2\mp B_0}}f\big\|_{L^\infty(\R^2)}\\&\leq C_T 2^{2j}\big(1+2^j|t|\big)^{-\frac12}\|\varphi(2^{-j}\sqrt{H_{B_0}}) f\|_{L^1(\R^2)}.
\end{split}
\end{equation}
\end{proposition}

\begin{remark} The finite $T$ can be chosen beyond  $\frac{\pi}{B_0}$. If we could prove \eqref{est: mic-decay2},
then \eqref{dis-w} follows
\begin{equation*}
\begin{split}
&\big\|e^{it\sqrt{H_{B_0}+m^2\mp B_0}}f\big\|_{L^\infty(\R^2)}\leq \sum_{j\in\Z}\big\|\varphi(2^{-j}\sqrt{H_{B_0}})e^{it\sqrt{H_{B_0}+m^2\mp B_0}}f\big\|_{L^\infty(\R^2)}\\
&\leq C_T |t|^{-\frac12}\sum_{j\in\Z}2^{\frac32j}\|\varphi(2^{-j}\sqrt{H_{B_0}}) f\|_{L^1(\R^2)}\leq C_T |t|^{-\frac12}\|f\|_{\dot{\mathcal{B}}^{3/2}_{1,1}(\R^2)}.
\end{split}
\end{equation*}
\end{remark}

\begin{proof}[The proof of Proposition \ref{prop: decay}]We estimate  the microlocalized half-wave propagator $$ \big\|\varphi(2^{-j}\sqrt{H_{B_0}})e^{it\sqrt{H_{B_0}+m^2\mp B_0}}f\big\|_{L^\infty(\R^2)}$$
by considering two cases that: $|t|2^j\gg 1$ and $|t|2^{j}\lesssim 1$. In the following argument, we can choose $\tilde{\varphi}\in C_c^\infty((0,+\infty))$ such that $\tilde{\varphi}(\lambda)=1$ if $\lambda\in\mathrm{supp}\,\varphi$
and $\tilde{\varphi}\varphi=\varphi$. Since $\tilde{\varphi}$ has the same property of $\varphi$, without confusion, we drop off the tilde above $\varphi$ for brief. Without loss of generality, in the following argument, we assume $t>0$.
\vspace{0.2cm}

{\bf Case 1: $t2^{j}\lesssim 1$.} We remark that we consider $t2^{j}\lesssim 1$ while not $t2^{j}\leq 1$, this will be used to extend the time interval. By the spectral theorem, one has
$$\|e^{it\sqrt{H_{B_0}+m^2\mp B_0}}\|_{L^2(\R^2)\to L^2(\R^2)}\leq C.$$
Indeed, denote $\lambda_{k,\ell}$ the spectrum of $H_{B_0}$, by the functional calculus, for $f\in L^2$,
we can write
\begin{equation*}
\begin{split}
e^{it\sqrt{H_{B_0}+m^2\mp B_0}} f&=\sum_{k\in\Z, \atop \ell\in\mathbb{N}} e^{it\sqrt{\lambda_{k,\ell}+m^2\mp B_0}} c_{k,\ell}\tilde{V}_{k,\ell}(x).
\end{split}
\end{equation*}
where
\begin{equation*}
\begin{split}
c_{k,\ell}=\int_{\mathbb{R}^2}f(y)\overline{\tilde{V}_{k,\ell}(y)} dy.
\end{split}
\end{equation*}

Thus we obtain
\begin{equation*}
\begin{split}
&\|e^{it\sqrt{H_{B_0}+m^2\mp B_0}} f\|_{L^2(\R^2)}\\
&=\Big(\sum_{k\in\Z, \atop \ell\in\mathbb{N}} \big| e^{it\sqrt{\lambda_{k,\ell}+m^2\mp B_0}} c_{k,\ell}\big|^2\Big)^{1/2}
=\Big(\sum_{k\in\Z, \atop \ell\in\mathbb{N}} \big| c_{k,\ell}\big|^2\Big)^{1/2}
=\|f\|_{L^2(\R^2)}.
\end{split}
\end{equation*}
Together with this, we use the Bernstein inequality \eqref{est:Bern}
 to prove
\begin{equation*}
\begin{split}
&\big\|\varphi(2^{-j}\sqrt{H_{B_0}})e^{it\sqrt{H_{B_0}+m^2\mp B_0}}f\big\|_{L^\infty(\R^2)}\\
&\lesssim 2^{j}\|e^{it\sqrt{H_{B_0}+m^2\mp B_0}}\varphi(2^{-j}\sqrt{H_{B_0}}) f\|_{L^2(\R^2)}\\
&\lesssim 2^{j}\|\varphi(2^{-j}\sqrt{H_{B_0}}) f\|_{L^2(\R^2)}\lesssim 2^{2j}\|\varphi(2^{-j}\sqrt{H_{B_0}}) f\|_{L^1(\R^2)}.
\end{split}
\end{equation*}
In this case $0<t\lesssim 2^{-j}$, we have
\begin{equation}\label{<1}
\begin{split}
&\big\|\varphi(2^{-j}\sqrt{H_{B_0}})e^{it\sqrt{H_{B_0}+m^2\mp B_0}}f\big\|_{L^\infty(\R^2)}\\
&\quad\lesssim 2^{2j}(1+2^jt)^{-N}\|\varphi(2^{-j}\sqrt{H_{B_0}}) f\|_{L^1(\R^2)},\quad \forall N\geq 0.
\end{split}
\end{equation}\vspace{0.2cm}

{\bf Case 2: $t2^{j}\gg 1$.} In this case, we can use
 \eqref{key-operator} to obtain the micolocalized
half-wave propagator
\begin{equation*}
\begin{split}
&\varphi(2^{-j}\sqrt{H_{B_0}})e^{it\sqrt{H_{B_0}+m^2\mp B_0}}\\&=\varphi(2^{-j}\sqrt{H_{B_0}})\tilde{\rho}\big(\frac{t(H_{B_0}+m^2\mp B_0)}{2^j}, 2^jt\big)\\
&+\varphi(2^{-j}\sqrt{H_{B_0}})\big(2^jt\big)^{\frac12}\int_0^\infty \chi(s,2^jt)e^{\frac{i2^jt}{4s}}e^{i2^{-j}ts(H_{B_0}+m^2\mp B_0)}\,ds.
\end{split}
\end{equation*}
We first use the spectral theorems and the Bernstein inequality again to estimate
\begin{equation*}
\begin{split}
&\big\|\varphi(2^{-j}\sqrt{H_{B_0}})\tilde{\rho}\big(\frac{t(H_{B_0}+m^2\mp B_0)}{2^j}, 2^jt\big) f\big\|_{L^\infty(\R^2)}.
\end{split}
\end{equation*}
On the one hand, by the support of $\varphi$, one has $2^{2j-2}\leq \lambda_{k,\ell}\leq 2^{2j+2}$. On the other hand,
we note the condition $2^jt\gg 1$ and $0<t\leq T$ with finite $T$, there exist a small constant $c$ depending on $T$ such that $\Big|\frac{m^2\mp B_0}{2^{2j}}\Big|\leq c\leq \frac1{100}$.
Therefore, we obtain
$$\frac18\leq \frac14+\frac{m^2\mp B_0}{2^{2j}}\leq \frac{t(\lambda_{k,\ell}+m^2\mp B_0)}{2^j}/2^jt=\frac{\lambda_{k,\ell}+m^2\mp B_0}{2^{2j}}\leq 4+\frac{m^2\mp B_0}{2^{2j}}\leq 8$$
which satisfies \eqref{est:rho}, then $$\big|\tilde{\rho}\big(\frac{t(\lambda_{k,m}+m^2\mp B_0)}{2^j}, 2^jt\big)\big|\leq C(1+2^jt)^{-N},\quad \forall N\geq 0.$$
Therefore, we use  the Bernstein inequality and  the spectral theorems to show
\begin{equation*}
\begin{split}
&\big\|\varphi(2^{-j}\sqrt{H_{B_0}})\rho\big(\frac{tH_{B_0}+m^2\mp B_0}{2^j}, 2^jt\big) f\big\|_{L^\infty(\R^2)}\\
&\lesssim 2^{j}\Big\|\rho\big(\frac{t(H_{B_0}+m^2\mp B_0)}{2^j}, 2^jt\big)\varphi(2^{-j}\sqrt{H_{B_0}}) f\Big\|_{L^2(\R^2)}\\
&\lesssim 2^{j}(1+2^jt)^{-N}\Big\|\varphi(2^{-j}\sqrt{H_{B_0}}) f\Big\|_{L^2(\R^2)}\\
&\lesssim 2^{2j}(1+2^jt)^{-N}\Big\|\varphi(2^{-j}\sqrt{H_{B_0}}) f\Big\|_{L^1(\R^2)}.
\end{split}
\end{equation*}

Next we use the dispersive estimates of Schr\"odinger propagator in \cite{KL}
\begin{equation*}
\begin{split}
&\big\|e^{it(H_{B_0}+m^2\mp B_0)} f\big\|_{L^\infty(\R^2)}\\
&=\big\|e^{itH_{B_0}} f\big\|_{L^\infty(\R^2)}\leq C |\sin(tB_0)|^{-1}\big\| f\big\|_{L^1\R^2)},\quad t\neq \frac{k\pi}{B_0}, \,k\in \Z,
\end{split}
\end{equation*}
 to estimate
\begin{equation*}
\begin{split}
&\big\|\varphi(2^{-j}\sqrt{H_{B_0}})\big(2^jt\big)^{\frac12}\int_0^\infty \chi(s,2^jt)e^{\frac{i2^jt}{4s}}e^{i2^{-j}ts(H_{B_0}+m^2\mp B_0)}f\,ds \big\|_{L^\infty(\R^2)}.
\end{split}
\end{equation*}
For $0<t<T_0<\frac{\pi}{2B_0}$, thus $\sin(tB_0)\sim tB_0$, then we obtain
\begin{equation*}
\begin{split}
&\big\|\varphi(2^{-j}\sqrt{H_{B_0}})\big(2^jt\big)^{\frac12}\int_0^\infty \chi(s,2^jt)e^{\frac{i2^jt}{4s}}e^{i2^{-j}ts(H_{B_0}+m^2\mp B_0)}f\,ds \big\|_{L^\infty(\R^2)}\\
&\lesssim\big(2^jt\big)^{\frac12}\int_0^\infty \chi(s,2^jt)|\sin(2^{-j}ts B_0)|^{-1}\,ds \big\|\varphi(2^{-j}\sqrt{H_{B_0}})f\big\|_{L^1(\R^2)}.
\end{split}
\end{equation*}
Since $s\in [\frac1{16}, 8]$ (the compact support of $\chi$ in $s$) and $B_0>0$,  if $2^{-j}t\leq \frac{\pi}{8B_0}$, then
\begin{equation}\label{>1}
\begin{split}
&\big\|\varphi(2^{-j}\sqrt{H_{B_0}})\big(2^jt\big)^{\frac12}\int_0^\infty \chi(s,2^jt)e^{\frac{i2^jt}{4s}}e^{i2^{-j}ts(H_{B_0}+m^2\mp B_0)}f\,ds \big\|_{L^\infty(\R^2)}\\
&\lesssim\big(2^jt\big)^{\frac12}(2^{-j}t)^{-1}\int_0^\infty \chi(s,2^jt)\,ds \big\|\varphi(2^{-j}\sqrt{H_{B_0}})f\big\|_{L^1(\R^2)}\\
&\lesssim 2^{2j}\big(2^jt\big)^{-\frac12}\big\|\varphi(2^{-j}\sqrt{H_{B_0}})f\big\|_{L^1(\R^2)}\\
&\lesssim 2^{2j}\big(1+2^jt\big)^{-\frac12}\big\|\varphi(2^{-j}\sqrt{H_{B_0}})f\big\|_{L^1(\R^2)}.
\end{split}
\end{equation}
Collecting \eqref{<1} and \eqref{>1}, it gives \eqref{est: mic-decay1}. To prove \eqref{est: mic-decay2}, we consider $0<t<T$.
For any $T>0$, there exists $j_0$ such that $2^{-j_0}T\leq \frac{\pi}{8B_0}$ with $j_0\in\Z_+$. For $j\leq j_0$, then
$2^j t\lesssim 1$, then one has \eqref{est: mic-decay2} from the first case. While for $j\geq j_0$,  if  $2^j t\lesssim 1$, one still has \eqref{est: mic-decay2} from the first case.
Otherwise, i.e. $2^j t\geq 1$, one has \eqref{est: mic-decay2} from the second case, since we always have $2^{-j}t\leq \frac{\pi}{8B_0}$ for $j\geq j_0$ and $0<t\leq T$.

\end{proof}

\subsection{Strichartz estimate}\label{sec:str}
In this section, we prove the Strichartz estimates \eqref{stri:w} in Theorem \ref{thm:dispersive} by using \eqref{est: mic-decay2}.
To this end, we need a variety of the abstract Keel-Tao's Strichartz estimates
theorem (\cite{KT}).
\begin{proposition}\label{prop:semi}
Let $(X,\mathcal{M},\mu)$ be a $\sigma$-finite measured space and
$U: I=[0,T]\rightarrow B(L^2(X,\mathcal{M},\mu))$ be a weakly
measurable map satisfying, for some constants $C$ may depending on $T$, $\alpha\geq0$,
$\sigma, h>0$,
\begin{equation}\label{md}
\begin{split}
\|U(t)\|_{L^2\rightarrow L^2}&\leq C,\quad t\in \mathbb{R},\\
\|U(t)U(s)^*f\|_{L^\infty}&\leq
Ch^{-\alpha}(h+|t-s|)^{-\sigma}\|f\|_{L^1}.
\end{split}
\end{equation}
Then for every pair $q,p\in[1,\infty]$ such that $(q,p,\sigma)\neq
(2,\infty,1)$ and
\begin{equation*}
\frac{1}{q}+\frac{\sigma}{p}\leq\frac\sigma 2,\quad q\ge2,
\end{equation*}
there exists a constant $\tilde{C}$ only depending on $C$, $\sigma$,
$q$ and $r$ such that
\begin{equation*}
\Big(\int_{I}\|U(t) u_0\|_{L^r}^q dt\Big)^{\frac1q}\leq \tilde{C}
\Lambda(h)\|u_0\|_{L^2}
\end{equation*}
where $\Lambda(h)=h^{-(\alpha+\sigma)(\frac12-\frac1p)+\frac1q}$.
\end{proposition}

\begin{proof}
This is an analogue of the semiclassical Strichartz
estimates for Schr\"odinger in \cite{KTZ, Zworski}. We refer to \cite{Zhang} for the proof.
\end{proof}

Now we prove the Strichartz estimates \eqref{stri:w}. Recall  $\varphi$ in \eqref{LP-dp} and
Littlewood-Paley frequency cutoff $\varphi_k(\sqrt{H_{B_0}})$, for each $k\in\Z$, we
define
\begin{equation*}
u_k(t,\cdot)=\varphi_k(\sqrt{H_{B_0}})u^i(t,\cdot), \quad i=1,2,
\end{equation*}
where $u^i(t,x)$ is the solution of \eqref{equ:W1} and \eqref{equ:W2} respectively.
Then, for each $k\in\Z$,  $u_k(t,x)$ solves the Cauchy problem
\begin{equation}\label{leq}
\partial_{t}^2u_k+(H_{B_0}+m^2\mp B_0) u_k=0, \quad u_k(0)=f_k(z),
~\partial_tu_k(0)=g_k(z),
\end{equation}
where $f_k=\varphi_k(\sqrt{H_{B_0}})u^i(0,x)$ and
$g_k=\varphi_k(\sqrt{H_{B_0}})\partial_tu^i(0,x)$. Since $(q, p)\in \Lambda_s^W$ in definition \ref{ad-pair}, then $q, p\geq2$. Thus,  by using the square-function
estimates \eqref{square} and the Minkowski inequality, we obtain
\begin{equation}\label{LP}
\|u^i(t,x)\|_{L^q(I;L^p(\R^2))}\lesssim
\Big(\sum_{k\in\Z}\|u_k(t,x)\|^2_{L^q(I;L^p(\R^2))}\Big)^{\frac12},\quad i=1,2,
\end{equation}
where $I=[0,T]$.
Denote the half-wave propagator $U(t)=e^{it\sqrt{H_{B_0}+m^2\mp B_0}}$, then
we write
\begin{equation}\label{sleq}
\begin{split}
u_k(t,z)
=\frac{U(t)+U(-t)}2f_k+\frac{U(t)-U(-t)}{2i\sqrt{H_{B_0}+m^2\mp B_0}}g_k.
\end{split}
\end{equation}
By using \eqref{LP} and \eqref{sleq}, we complete the proof of \eqref{stri:w} after taking summation in $k\in\Z$ if we could prove
\begin{proposition}\label{lStrichartz} Let
$f_k=\varphi_k(\sqrt{H_{B_0}})f$ for $\varphi_k$ in \eqref{LP-dp} and $k\in\Z$. Then
\begin{equation}\label{lstri}
\|U(t)f_k\|_{L^q(I;L^p(\R^2))}\leq C_T
2^{ks}\|f\|_{L^2(\R^2)},
\end{equation}
where the admissible pair $(q,p)\in [2,+\infty]\times [2,+\infty)$ and $s$ satisfy
\eqref{adm} and \eqref{scaling}.
\end{proposition}

\begin{proof}
Since $f_k=\varphi_k(\sqrt{H_{B_0}})f$, then $$U(t)f_k=\varphi_k(\sqrt{H_{B_0}})e^{it\sqrt{H_{B_0}+m^2\mp B_0}} f:=U_k f.$$
By using the spectral theorem, we see
\begin{equation*}
\|U_k(t)f\|_{L^2(\R^2)}\leq C\|f\|_{L^2(\R^2)}.
\end{equation*}
By using \eqref{est: mic-decay2}, we obtain
\begin{equation*}
\begin{split}
\|U_k(t)U_k^*(s)f\|_{L^\infty (\R^2)}&=\|U_k(t-s)f\|_{L^\infty (\R^2)}\\
&\leq C_T 2^{\frac32 k}\big(2^{-k}+|t-s|\big)^{-\frac12}\|f\|_{L^1(\R^2)},
\end{split}
\end{equation*}
Then the estimates \eqref{md}
for $U_{k}(t)$ hold for $\alpha=3/2$, $\sigma=1/2$ and
$h=2^{-k}$. Hence, Proposition \ref{prop:semi} gives
\begin{equation*}
\|U(t)f_k\|_{L^q(I;L^p(\R^2))}=\|U_k(t)f\|_{L^q(I;L^p(\R^2))}\leq C_T
2^{k[2(\frac12-\frac1p)-\frac1q]} \|f\|_{L^2(\R^2)}.
\end{equation*}
which implies \eqref{lstri} since $s=2(\frac12-\frac1p)-\frac1q$.
\end{proof}

Then, due to the well known decay estimates for solutions to the wave equation (see \cite{SS}), we deduce
 \begin{align*}
 \|u(t,x)\|_{[L^q(I;L^p(\R^2))]^2}\leq C_T \|f\|_{[\dot{\mathcal{H}}^s_{\A}(\R^2)]^2}+\|-i\D_A f\|_{[\dot{\mathcal{H}}^{s-1}_{\A}(\R^2)]^2},
 \end{align*}
hence by Lemma \ref{lem:D-norm} we obtain
 \begin{equation}
 \begin{aligned}
 \|u(t,x)\|_{[L^q(I;L^p(\R^2))]^2}\leq C_T \|f\|_{[\mathcal{H}^s_{\A}(\R^2)]^2}.
 \end{aligned}
 \end{equation}
Then we conclude the proof of Theorem \ref{thm:dispersive}.

\begin{center}

\end{center}


\begin{thebibliography}{99}

\bibitem{Abra} {\sc M. Abramowitz, I. Stegun eitors}, \textit{Handbook of Mathematical Functions}. Dover Publications, New York, (1970).

\bibitem{AB59} {\sc Y. Aharonov and D. Bohm}, \textit{ Significance of electromagnetic potentials in the quantum theory}, Phys. Rev. 115 (1959), no. 2, 485-491.

\bibitem{AHS1}  {\sc J. Avron, I. Herbst and B. Simon}, \textit{ Schr\"odinger operators with magnetic fields. I. General interactions}, Duke Math. J. 45 (1978), no. 4, 847-883.

\bibitem{AHS2}  {\sc J. Avron, I. Herbst and B. Simon}, \textit{Schr\"odinger operators with magnetic fields, II. Separation of center of mass in homogeneous magnetic fields}, Ann. Phys. 114 (1978), 431-451.

\bibitem{AHS3}  {\sc J. Avron, I. Herbst and B. Simon}, \textit{Schr\"odinger operators in magnetic fields, III. Atoms in homogeneous magnetic field}, Commun. Math. Phys. 79 (1981), 529-572.

\bibitem{AS65} {\sc M. Abramowitz and I. A. Stegun}, \textit{ Handbook of mathematical functions with formulas, graphs and mathematical tables}, U. S. Government Printing Office, Washington, DC, 1965.


\bibitem{CDYZ} {\sc F. Cacciafesta, P. D'Ancona, Z. Yin and J. Zhang}, \textit{Dispersive estimates for massless Dirac equation in Aharonov-Bohm magnetic fields}, In preparetion.


\bibitem{cacser} {\sc F. Cacciafesta, E. S\`{e}r\`{e}}, \textit{Local smoothing estimates for the massless Dirac-Coulomb equation in 2 and 3 dimensions}, J. Funct. Anal. 271 (2016), no. 8, 2339-2358.

\bibitem{cacserzha}
{\sc F. Cacciafesta, E. S\`{e}r\`{e} and J. Zhang}, \textit{ Generalized Strichartz estimates for the massless Dirac-Coulomb equation}, Springer INdAM Ser. 52 (2022), 127-139

\bibitem{cacfan} \sc{F. Cacciafesta, L. Fanelli}, \textit{ Dispersive estimates for the Dirac equation in an Aharonov-Bohm field}, J. Differential equations. 263 (2017), no. 7, 4382-4399.

\bibitem{cacfan2} {\sc F. Cacciafesta, L. Fanelli}, \textit{ Weak dispersive estimates for fractional Aharonov-Bohm-Schroedinger groups}, Dynamics of PDE. Vol. 16 (2019) no. 1, 95-103.

\bibitem{CYZ} {\sc F. Cacciafesta, Z. Yin, J. Zhang}, \textit{Generalized Strichartz estimates for the wave and Dirac equations in Aharonov-Bohm magnetic fields}, Dynamics of PDE. 19 (2022), no. 1, 71-90.

\bibitem{CS} {\sc S. Cuccagna and P. P. Schirmer}, \textit{ On the wave equation with a magnetic potential}, Comm. Pure Appl. Math. 54 (2001), no. 2, 135-152.

\bibitem{DF} {\sc P. D'Ancona and L. Fanelli}, \textit{ Decay estimates for the wave and Dirac equations with a magnetic potential}, Comm. Pure Appl. Math. 60 (2007), no. 3, 357-392.

\bibitem{DFVV} {\sc P. D'Ancona, L. Fanelli, L. Vega and N. Visciglia},\textit{ Endpoint Strichartz estimates for the magnetic Schr\"odinger equation}, J. Funct. Anal. 258 (2010), no. 10, 3227-3240.

\bibitem{DPR} {\sc P. D'Ancona, V. Pierfelice and F. Ricci},\textit{ On the wave equation associated to the Hermite and the twisted Laplacian}, J. Fourier Anal. Appl. 16 (2010), no. 2, 294-310.

\bibitem{D} {\sc E. Danesi}, \textit{Stricartz estimates for the 2D and 3D massless Dirac-Coulomb equations and applications}, J. Funct. Anal. 286 (2024), no. 3, 110251.

\bibitem{EGS1} {\sc M. B. Erdo\u{g}an, M. Goldberg and W. Schlag},\textit{ Strichartz and smoothing estimates for Schr\"odinger operators with almost critical magnetic potentials in three and higher dimensions}, Forum Math. 21 (2009), no. 4, 687-722.

\bibitem{EGS2} {\sc M. B. Erdogan, M. Goldberg and W. Schlag}, \textit{ Strichartz and smoothing estimates for Schr\"odinger operators with large magnetic potentials in $\R^3$}, J. Eur. Math. Soc. 10 (2008), no. 2, 507-531.

\bibitem{EG}  {\sc M. B. Erdo\u{g}an, W. R. Green}, \textit{The Dirac equation in two dimensions: dispersive estimates and classification of threshold obstructions}, Commun. Math. Phys. 352(2017), 719-757.

\bibitem{EGG}  {\sc M. B. Erdo\u{g}an, M. Goldberg, and W. R. Green}, \textit{ Limiting absorption principle and Strichartz estimates for Dirac operators in two and higher dimensions},  Commun. Math. Phys. 367(2019), 241-263.
%
\bibitem{EGT}  {\sc M. B. Erdo\u{g}an, W. R. Green,  E. Toprak}, \textit{ Dispersive estimates for Dirac operators in dimension three with obstructions at threshold energies},  Amer. J. Math. 141(2019), 1217-1258.


\bibitem{FP} {\sc H. Falomir, and P. A. G. Pisani}, \textit{Hamiltonian self-adjoint extensions for (2+1)-dimensional Dirac particles}, Journal of Physics A: Mathematical and General 34 (2001), no. 19, 4143.

\bibitem{F} {\sc L. Fanelli}, \textit{ Spherical Schr\"odinger Hamiltonians: spectral analysis and time decay}, A. Michelangeli, G. DellAntonio (eds.), Advances in Quantum Mechanics, Springer INdAM Ser., 18, Springer, Cham, 2017.

\bibitem{FFFP} {\sc L. Fanelli, V. Felli, M. A. Fontelos and A. Primo},\textit{ Time decay of scaling invariant electromagnetic Schr\"odinger equations on the plane}, Comm. Math. Phys. 337 (2015), no. 3, 1515-1533.

\bibitem{FFFP1} {\sc L. Fanelli, V. Felli, M. A. Fontelos and A. Primo}, \textit{ Time decay of scaling critical electromagnetic Schr\"odinger flows}, Comm. Math. Phys. 324 (2013), no. 3, 1033-1067.

\bibitem{FV} {\sc L. Fanelli,  and L. Vega}, \textit{ Magnetic virial identities, weak dispersion and Strichartz inequalities}, Math. Ann. 344 (2009), 249-278.

\bibitem{FZZ} {\sc L. Fanelli, J. Zhang and J. Zheng}, \textit{ Dispersive estimates for 2D-wave equations with critical potentials}, Adv. Math. 400 (2022), Paper No. 108333, 46 pp.

\bibitem{GYZZ22}   {\sc X. Gao, Z. Yin, J. Zhang and J. Zheng}, \textit{ Decay and Strichartz estimates in critical electromagnetic fields}, J. Funct. Anal. 282 (2022), no. 5, Paper No. 109350, 51 pp.

\bibitem{GV} {\sc  J. Ginibre, G. Velo}, \textit{ Generalized Strichartz inequalities for the wave equation}, J. Funct. Anal. 133 (1995), 50-68.


\bibitem{KT} {\sc M. Keel and T. Tao}, \textit{ Endpoint Strichartz estimates}, Amer. J. Math. 120 (1998), no. 5, 955-980.

\bibitem{KL} {\sc T. F. Kieffer and M. Loss}, \textit{ Non-linear Schr\"odinger equation in a uniform magnetic field}, Partial differential equations, spectral theory, and mathematical physics - the Ari Laptev anniversary volume, EMS Ser. Congr. Rep. EMS Press, Berlin, (2021), 247-265.

\bibitem{KTZ} {\sc H. Koch, D. Tataru and M. Zworski}, \textit{ Semiclassical $L^p$ estimates}, Ann. Henri Poincar\'{e}, 8 (2007), 885-916.

\bibitem{LL} {\sc L.D. Landau, E.M. Lifshitz}, \textit{Course of theoretical physics, Volume 4: Quantum electrodynamics}, Pergamon Press, 1987.

\bibitem{LS} {\sc H. Leinfelder, C.G. Simader}, \textit{ Schr\"{o}dinger operators with singular magnetic vector potentials}, Math. Z. 176(1981), 1-19.

\bibitem{LT} {\sc M. Loss and B. Thaller}, \textit{ Optimal heat kernel estimates for Schr\"odinger operators with magnetic fields in two dimensions}, Comm. Math. Phys. 186 (1997), no. 1, 95-107.

\bibitem{MS} {\sc D. M\"uller and A. Seeger},\textit{ Sharp  $L^p$  bounds for the wave equation on groups of Heisenberg type}, Anal. PDE 8 (2015), no. 5, 1051-1100.

\bibitem{RS} {\sc M. Reed and B. Simon}, \textit{ Methods of modern mathematical physics. II. Fourier analysis, self-adjointness}, Academic Press, New York-London, 1975.


\bibitem{S} {\sc W. Schlag}, \textit{ Dispersive estimates for Schr\"odinger operators: a survey}, Mathematical aspects of nonlinear dispersive equations, Ann. of Math. Stud. Princeton Univerity Press, Princeton, NJ, 163 (2007), 255-285.

\bibitem{SS} {\sc J. Shatah, M. Struwe}, \textit{Geometric wave equations. New York University Courant Institute of Mathematical Sciences}, New York, 1998.

\bibitem{Simon2} {\sc B. Simon}, \textit{ Functional integration and quantum physics}, Pure Appl. Math., 86, Academic Press, New York-London, 1979.

\bibitem{Stein} {\sc E. M. Stein}, \textit{ Singular integrals and differentiability properties of functions}, Princeton Math. Ser., No. 30, Princeton University Press, Princeton, NJ, 1970.


\bibitem{WZZ1} {\sc H. Wang, F. Zhang and J. Zhang}, \textit{ Decay estimates for one Aharonov-Bohm solenoid in a uniform magnetic field I: Schr\"odinger equation}, arXiv 2309.07635.

\bibitem{WZZ2}  H. Wang, F. Zhang and J. Zhang, \textit{ Decay estimates for one Aharonov-Bohm solenoid in a uniform magnetic field II:  wave equation}, arXiv 2309.07649.

\bibitem{Zhang} {\sc J. Zhang}, \textit{ Strichartz estimates and nonlinear wave equation on nontrapping asymptotically conic manifolds}, Adv. Math. 271 (2015), 91-111.

 \bibitem{Zworski} {\sc M. Zworski}, \textit{ Semiclassical analysis}, Grad. Stud. Math., 138, American Mathematical Society, Providence, RI, 2012.

\end{thebibliography}
\end{document}